\documentclass[11pt]{article}

\usepackage{amsmath, amsfonts, amssymb, amsthm,  graphicx, mathtools, enumerate,amsbsy, dsfont}

\usepackage[blocks, affil-it]{authblk}

\usepackage{mathrsfs}

\usepackage[numbers, square, sort]{natbib}
\usepackage[CJKbookmarks=true,
            bookmarksnumbered=true,
			bookmarksopen=true,
			colorlinks=true,
			citecolor=red,
			linkcolor=blue,
			anchorcolor=red,
			urlcolor=blue]{hyperref}
\usepackage[usenames]{color}

\usepackage[letterpaper, left=1.2truein, right=1.2truein, top = 1.2truein, bottom = 1.2truein]{geometry}
\usepackage[ruled, vlined, lined, commentsnumbered]{algorithm2e}

\usepackage{prettyref,soul}

\newtheorem{lemma}{Lemma}[section]

\newtheorem{thm}{Theorem}[section]

\newtheorem{corollary}{Corollary}[section]

\newrefformat{eq}{(\ref{#1})}
\newrefformat{chap}{Chapter~\ref{#1}}
\newrefformat{sec}{Section~\ref{#1}}
\newrefformat{algo}{Algorithm~\ref{#1}}
\newrefformat{fig}{Fig.~\ref{#1}}
\newrefformat{tab}{Table~\ref{#1}}
\newrefformat{rmk}{Remark~\ref{#1}}
\newrefformat{clm}{Claim~\ref{#1}}
\newrefformat{def}{Definition~\ref{#1}}
\newrefformat{cor}{Corollary~\ref{#1}}
\newrefformat{lmm}{Lemma~\ref{#1}}
\newrefformat{lemma}{Lemma~\ref{#1}}
\newrefformat{prop}{Proposition~\ref{#1}}
\newrefformat{app}{Appendix~\ref{#1}}
\newrefformat{ex}{Example~\ref{#1}}
\newrefformat{exer}{Exercise~\ref{#1}}
\newrefformat{soln}{Solution~\ref{#1}}
\newrefformat{cond}{Condition~\ref{#1}}

\def\Var{\textsf{Var}} 

\def\text#1{\mbox{\rm #1}}

\def\T{\mathcal{T}}

\def\T{{ \mathrm{\scriptscriptstyle T} }}
\def\H{{ \mathrm{\scriptscriptstyle H} }}
\def\n{{ \mathcal{N} }}
\def\cn{{ \mathcal{CN} }}

\newcommand{\argmin}{\mathop{\rm argmin}}
\newcommand{\argmax}{\mathop{\rm argmax}}

\newcommand{\norm}[1]{\left\|{#1} \right\|}

\newcommand{\wh}{\widehat}
\newcommand{\wt}{\widetilde}

\newcommand{\fnorm}[1]{\|#1\|_{\rm F}}
\newcommand{\opnorm}[1]{\|#1\|_{\rm op}}

\newcommand{\Tr}{\mathop{\sf Tr}}
\newcommand{\diag}{\mathop{\text{diag}}}

\newtheorem*{condb'}{Condition B'}

\newcommand{\br}[1]{\left( #1 \right)}
\newcommand{\sbr}[1]{\left[ #1 \right]}
\newcommand{\cbr}[1]{\left\{ #1 \right\}}
\newcommand{\ebr}[1]{\exp\left( #1 \right)}

\newcommand{\mathr}{\mathbb{R}}

\newcommand{\abs}[1]{\left| #1 \right|}

\newcommand{\E}{\mathbb{E}}

\newcommand{\im}{{\rm Im}}
\newcommand{\re}{{\rm Re}}

\newcommand{\diff}{{\rm d}}

\title{Exact Minimax Estimation for Phase Synchronization
}
\author[1]{Chao Gao}
\author[2]{Anderson Y. Zhang}
\affil[1]{
University of Chicago
}
\affil[2]{
University of Pennsylvania
}

\begin{document}
\maketitle

\begin{abstract}
We study the phase synchronization problem with measurements $Y=z^*z^{*\H}+\sigma W\in\mathbb{C}^{n\times n}$, where $z^*$ is an $n$-dimensional complex unit-modulus vector and $W$ is a complex-valued Gaussian random matrix. It is assumed that each entry $Y_{jk}$ is observed with probability $p$. We prove that the minimax lower bound of estimating $z^*$ under the squared $\ell_2$ loss is $(1-o(1))\frac{\sigma^2}{2p}$. We also show that both generalized power method and maximum likelihood estimator achieve the error bound $(1+o(1))\frac{\sigma^2}{2p}$. Thus, $\frac{\sigma^2}{2p}$ is the exact asymptotic minimax error of the problem. Our upper bound analysis involves a precise characterization of the statistical property of the power iteration. The lower bound is derived through an application of van Trees' inequality.

\smallskip

\end{abstract}


\section{Introduction}

The phase synchronization problem \citep{singer2011angular} is to estimate $n$ unknown angles $\theta_1^*,\cdots,\theta_n^*$ from noisy measurements of $(\theta_j^*-\theta_k^*)$ mod $2\pi$. In this paper, we consider the following additive model \cite{abbe2017group, gao2020multi, zhong2018near, ling2020near}:
\begin{equation}
Y_{jk}=z_j^*\bar{z}_k^*+\sigma W_{jk}\in\mathbb{C}, \label{eq:basic-model}
\end{equation}
for all $1\leq j<k\leq n$, where we use the notation $\bar{x}$ for the complex conjugate of $x$. We assume that each $z_j^*\in\mathbb{C}_1=\{x\in\mathbb{C}:|x|=1\}$ and we can thus write it as $z_j^*=e^{i\theta_j^*}$. The additive noise $W_{jk}$ in (\ref{eq:basic-model}) is assumed to be i.i.d. standard complex Gaussian.\footnote{For $W_{jk}\sim \cn(0,1)$, we have $\re(W_{jk})\sim \n\left(0,\frac{1}{2}\right)$ and $\im(W_{jk})\sim \n\left(0,\frac{1}{2}\right)$ independently.} Our goal in this paper is to study minimax optimal estimation of the vector $z^*\in\mathbb{C}_1^n$ under the loss function
\begin{equation}
\ell(\wh{z},z^*)=\min_{a\in\mathbb{C}_1}\sum_{j=1}^n|\wh{z}_j a-z_j^*|^2. \label{eq:loss}
\end{equation}
We remark that the minimization over a global phase in the definition of (\ref{eq:loss}) is necessary. This is because the global phase is not identifiable from the pairwise observations (\ref{eq:basic-model}).

Various estimation procedures have been considered and studied in the literature. For example, the maximum likelihood estimator (MLE) is defined as the global maximizer of the following constrained optimization problem
\begin{equation}
\max_{z\in\mathbb{C}_1^n}z^{\H}Yz,\label{eq:MLE-full-graph}
\end{equation}
where $Y$ is Hermitian with $Y_{jk}=\bar Y_{kj}$ for all $1\leq k <j \leq n$ and $Y_{jj}=0$ for  all $j\in[n]$. Note that  (\ref{eq:MLE-full-graph}) can be shown to be  equivalent to $\min_{z\in\mathbb{C}_1^n} \sum_{1\leq j<k\leq n} \abs{Y_{jk} - z_j\bar z_k}^2$.
It was shown in \cite{bandeira2017tightness} that the MLE satisfies $\ell(\wh{z},z^*)\leq C\sigma^2$ with high probability for some constant $C>0$. However, the optimization (\ref{eq:MLE-full-graph}) is nonconvex and computationally infeasible in general. To address this problem, generalized power method (GPM) \cite{boumal2016nonconvex, filbir2020recovery, perry2018message} and semi-definite programming (SDP) have been considered in the literature to approximate the solution of (\ref{eq:MLE-full-graph}). The generalized power method is defined through the iteration,\footnote{When the denominator of (\ref{eq:GPM-full-graph}) is zero, take $z_j^{(t)}$ to be an arbitrary value in $\mathbb{C}_1$.}
\begin{equation}
z_j^{(t)} = \frac{\sum_{k\in[n]\backslash\{j\}}Y_{jk}z_k^{(t-1)}}{\left|\sum_{k\in[n]\backslash\{j\}}Y_{jk}z_k^{(t-1)}\right|}. \label{eq:GPM-full-graph}
\end{equation}
In other words, one repeatedly computes the product $Yz^{(t-1)}$ and projects this vector to $\mathbb{C}_1^n$ through entrywise normalization. When the iteration (\ref{eq:GPM-full-graph}) is initialized by the eigenvector method, 
\cite{zhong2018near} shows that $z^{(t)}$ converges to the global maximizer of (\ref{eq:MLE-full-graph}) at a linear rate under the noise level condition 
$\sigma^2=O\left(\frac{n}{\log n}\right)$. For its statistical performance, \cite{liu2017estimation} shows $\ell(z^{(\infty)},z^*)\leq C\sigma^2$ with high probability for some constant $C>0$.
The semidefinite programming is a convex relaxation of (\ref{eq:MLE-full-graph}). It refers to the following optimization problem,
\begin{equation}
\max_{Z=Z^{\H}\in\mathbb{R}^{n\times n}}\Tr(YZ)\quad\text{subject to }\diag(Z)=I_n\text{ and }Z \succeq 0. \label{eq:SDP}
\end{equation}
In general, the solution of (\ref{eq:SDP}) is an $n\times n$ matrix and needs to be rounded. When $\sigma^2=O(n^{1/2})$, it was proved by \cite{bandeira2017tightness} that the solution to (\ref{eq:SDP}) is a rank-one matrix $\wh{Z}=\wh{z}\wh{z}^{\H}$, with $\wh{z}$ being a global maximizer of (\ref{eq:MLE-full-graph}). This result was recently proved by \cite{zhong2018near} to hold under a weaker condition $\sigma^2=O\left(\frac{n}{\log n}\right)$. Given the fact that SDP solves (\ref{eq:MLE-full-graph}), we know that it also achieves the same high-probability error bound $\ell(\wh{z},z^*)\leq C\sigma^2$ as that of the MLE under the additional condition $\sigma^2=O\left(\frac{n}{\log n}\right)$.

Despite these estimation procedures studied in the literature, it remains an open problem what the optimal error under the loss (\ref{eq:loss}) is. In this paper, we establish a minimax lower bound for phase synchronization. We show that
\begin{equation}
\inf_{\wh{z}\in\mathbb{C}_1^n}\sup_{z\in\mathbb{C}_1^n}\mathbb{E}_z\ell(\wh{z},z)\geq (1-\delta)\frac{\sigma^2}{2}, \label{eq:intro-lower-full}
\end{equation}
for some $\delta=o(1)$ under the condition that $\sigma^2=o(n)$. 
This provides a stronger characterization of the fundamental limits of the phase synchronization problem than the Cram\'er-Rao lower bound developed in  \cite{boumal2013intrinsic, boumal2014cramer}, which only holds for unbiased estimators. Instead, the lower bound in (\ref{eq:intro-lower-full}) holds for both unbiased and biased estimators.

Moreover, in this paper, we prove both the MLE and the GPM achieve the error bound
\begin{equation}
\ell(\wh{z},z^*)\leq (1+\delta)\frac{\sigma^2}{2}, \label{eq:intro-upper-full}
\end{equation}
for some $\delta=o(1)$ with high probability under the same condition $\sigma^2=o(n)$. In other words, these two estimators are not only rate-optimal, but are also exactly asymptotically minimax by achieving the correct leading constant in front of the optimal rate. In addition, since we know by the result of \cite{zhong2018near} that the solution of the SDP is a rank-one matrix $\wh{z}\wh{z}^{\H}$ with $\wh{z}$ being the MLE, the SDP also achieves the optimal error bound (\ref{eq:intro-upper-full}) as a direct consequence, but under a stronger condition $\sigma^2=O\left(\frac{n}{\log n}\right)$.

To formally state our main result, we introduce a more general statistical estimation setting that allows the possibility of missing entries. Instead of observing $Y_{jk}$ for all $1\leq j<k\leq n$, we assume each $Y_{jk}$ is observed with probability $p$. In other words, consider a random graph $A_{jk}\sim \text{Bernoulli}(p)$ independently for all $1\leq j<k\leq n$, and we only observe $Y_{jk}$ that follows (\ref{eq:basic-model}) when $A_{jk}=1$.  Define $A_{jk} = A_{kj}$ for $1\leq k<j\leq n$ and $A_{jj}=0$ for $j\in[n]$.
The full observations can be organized into two Hermitian matrices $A$ and $A\circ Y$, where $\circ$ denotes the matrix Hadamard product. The MLE and the GPM can be extended by replacing $Y_{jk}$ in (\ref{eq:MLE-full-graph}) and (\ref{eq:GPM-full-graph}) with $A_{jk}Y_{jk}$.

\begin{thm}\label{thm:intro}
Assume $\sigma^2=o(np)$ and $\frac{np}{\log n}\rightarrow\infty$. Then, there exists some $\delta=o(1)$ such that
\begin{align}
\inf_{\wh{z}\in\mathbb{C}_1^n}\sup_{z\in\mathbb{C}_1^n}\mathbb{E}_z\ell(\wh{z},z)\geq (1-\delta)\frac{\sigma^2}{2p}.\label{eqn:intro_1}
\end{align}
Moreover, both MLE and GPM achieve the error bound
\begin{align}
\ell(\wh{z},z^*)\leq (1+\delta)\frac{\sigma^2}{2p},\label{eqn:intro_2}
\end{align}
with probability at least $1-n^{-1}-\exp\left(-\left(\frac{np}{\sigma^2}\right)^{1/4}\right)$.
\end{thm}

Theorem \ref{thm:intro} immediately implies (\ref{eq:intro-lower-full}) and (\ref{eq:intro-upper-full}) as a special case of $p=1$, and is the first statistical analysis of phase synchronization for $p<1$. We remark that both conditions $\sigma^2=o(np)$ and $\frac{np}{\log n}\rightarrow\infty$ are essential for the results of the above theorem to hold. Since the minimax risk of the problem is $\frac{\sigma^2}{2p}$, the condition $\sigma^2=o(np)$, which is equivalent to $\frac{\sigma^2}{2p}=o(n)$, guarantees that the minimax risk is of smaller order than the trivial one. The order $n$ is trivial, since $\ell(z,z^*)\leq 4n$ for any $z,z^*\in\mathbb{C}_1^n$. When $p=1$, the necessity of $\sigma^2=o(n)$ for a nontrivial recovery is understood in the literature \citep{perry2016optimality,javanmard2016phase,zhong2018near,lelarge2019fundamental}. The condition $\frac{np}{\log n}\rightarrow\infty$ guarantees that the random graph $A$ is connected with high probability. It is known that when $p\leq c\frac{\log n}{n}$ for some sufficiently small constant $c>0$, the random graph has several disjoint components, which makes the recovery of $z^*$ up to a global phase impossible. 

Since $\ell(\wh{z},z)\leq 4n$ and $\exp\left(-\left(\frac{np}{\sigma^2}\right)^{1/4}\right)=o\left(\frac{\sigma^2}{np}\right)$, the MLE and the GPM also satisfy the in-expectation bound
$$\sup_{z\in \mathbb{C}_1^n}\mathbb{E}_z\ell(\wh{z},z)\leq (1+\delta)\frac{\sigma^2}{2p} + 4,$$
for some $\delta=o(1)$.
To understand how small the extra term $4$ is, this is the error achieved by an estimator that exactly recovers $n-1$ entries of $z$ but only makes mistake on one entry. For example, consider $\wh{z}_1=-z_1^*$ and $\wh{z}_j=z_j^*$ for all $j=2,\cdots,n$. This gives $\ell(\wh{z},z^*)=(1-o(1))4$. If we further assume that $p=o(\sigma^2)$, we will have
$$(1-\delta)\frac{\sigma^2}{2p}\leq \inf_{\wh{z}\in\mathbb{C}_1^n}\sup_{z\in\mathbb{C}_1^n}\mathbb{E}_z\ell(\wh{z},z)\leq (1+\delta)\frac{\sigma^2}{2p},$$
for some $\delta=o(1)$. Thus, $\frac{\sigma^2}{2p}$ is the exact asymptotic minimax risk for phase synchronization.

Our analysis of the upper bounds relies on a precise statistical characterization of the power iteration map $f:\mathbb{C}_1^n\rightarrow\mathbb{C}_1^n$. Let $f$ be the map that characterizes the iteration of the GPM. That is, $z^{(t)}=f(z^{(t-1)})$. We show that as long as $z\in\mathbb{C}_1^n$ satisfies $\ell(z,z^*)=o(n)$, the vector $f(z)$ must satisfy
\begin{equation}
\ell(f(z),z^*) \leq \delta \ell(z,z^*) + (1+\delta)\frac{\sigma^2}{2p}, \label{eq:critical-lem}
\end{equation}
for some $\delta=o(1)$ with high probability. To be more precise, we prove that the inequality (\ref{eq:critical-lem}) holds uniformly over all $z\in\mathbb{C}_1^n$ such that $\ell(z,z^*)\leq \gamma n$ for some $\gamma=o(1)$ with high probability.  The bound (\ref{eq:critical-lem}) immediately leads to the optimality of the GPM. This direct analysis of the power iteration is very different from what has been done in the literature. In the literature, the statistical error bound of the GPM is derived through its convergence to the MLE \citep{boumal2016nonconvex,zhong2018near}, but that requires a stronger condition $\sigma^2=O\left(\frac{n}{\log n}\right)$ at least when $p=1$. In contrast, our analysis of GPM is not based on its relation to the MLE. On the opposite, we analyze the MLE based on its relation to the GPM.
The optimality of the MLE can also be derived from (\ref{eq:critical-lem}). This is by showing that MLE is a fixed point of the map $f$. That is, $\wh{z}=f(\wh{z})$, and therefore with $z=\wh{z}$, we can rearrange (\ref{eq:critical-lem}) into the bound $\ell(\wh{z},z)\leq \frac{1+\delta}{1-\delta}\frac{\sigma^2}{2p}$. 

To derive the lower bound result, we show that it is sufficient to analyze the Bayes risk of a subproblem of estimating the relative angle $z_j\bar{z}_k$ for each $j\neq k$,
\begin{equation}
\inf_{\wh{T}}\int \int \pi(z_j)\pi(z_k)\mathbb{E}_z|\wh{T}-z_j\bar{z}_k|^2\diff z_j\diff z_k. \label{eq:pair-low-piece}
\end{equation}
In other words, the difficulty of estimating a vector in $\mathbb{C}_1^n$ is determined by the average difficulty of estimating $z_j\bar{z}_k$ given the knowledge of $(z_l)_{l\in[n]\backslash\{j,k\}}$ for all $j\neq k$. To lower bound (\ref{eq:pair-low-piece}), we apply a multivariate van Trees' inequality \citep{gill1995applications} that relates the Bayes risk (\ref{eq:pair-low-piece}) to the Fisher information of the phase synchronization model.

\paragraph{Paper Organization.} The rest of the paper is organized as follows. In Section \ref{sec:key}, we establish a key lemma that implies the critical inequality (\ref{eq:critical-lem}). The implications of the key lemma on the statistical error bounds of GPM and MLE are discussed in Section \ref{sec:upper}, which establishes (\ref{eqn:intro_2}) for Theorem \ref{thm:intro}. The minimax lower bound of phase synchronization is proved in Section \ref{sec:lower}, which establishes (\ref{eqn:intro_1}) for Theorem \ref{thm:intro}. Finally, Section \ref{sec:proof} collects the remaining technical proofs of the paper.

\paragraph{Notation.}

For $d \in \mathbb{N}$, we write $[d] = \{1,\dotsc,d\}$.  Given $a,b\in\mathbb{R}$, we write $a\vee b=\max(a,b)$ and $a\wedge b=\min(a,b)$. For a set $S$, we use $\mathbb{I}\{S\}$ and $|S|$ to denote its indicator function and cardinality respectively. For a complex number $x\in\mathbb{C}$, we use $\bar{x}$ for its complex conjugate and $\abs{x}$ for its modulus.
For a matrix $B =(B_{jk})\in\mathbb{C}^{d_1\times d_2}$, we use $B^\H \in\mathbb{C}^{d_2\times d_1}$ for its conjugate transpose such that $B^{\H}= (\bar{B}_{kj})$. 
The Frobenius norm and operator norm of $B$ are defined by $\fnorm{B}=\sqrt{\sum_{j=1}^{d_1}\sum_{k=1}^{d_2}|B_{jk}|^2}$ and $\opnorm{B} = \sup_{u\in\mathbb{C}^{d_1},v\in\mathbb{C}^{d_2}:\norm{u}=\norm{v}=1} u^\H Bv$. For $U,V\in\mathbb{C}^{d_1\times d_2}$, $U\circ V\in\mathbb{R}^{d_1\times d_2}$ is the Hadamard product $U\circ V=(U_{jk}V_{jk})$.
The notation $\mathbb{P}$ and $\mathbb{E}$ are generic probability and expectation operators whose distribution is determined from the context.

\section{A Key Lemma}\label{sec:key}

Recall that we observe a random graph $A_{jk}\sim \text{Bernoulli}(p)$ independently for all $1\leq j<k\leq n$. Whenever $A_{jk}=1$, we also observe $Y_{jk}=z_j^*\bar{z}_k^*+\sigma W_{jk}$ with $W_{jk}\sim\cn(0,1)$. In summary, the observations contain an adjacency matrix $A$ and a masked version of pairwise interactions $A\circ Y$, which are both Hermitian  as we define $Y_{jk}=\bar Y_{kj}$ and $A_{jk}= A_{kj}$ for all $1\leq k<j \leq n$ and $Y_{jj}=A_{jj}=0$ for all $j\in[n]$.

In this section, we establish a lemma that shows the contraction of the loss function through power iteration. Note that $\mathbb{E}(A\circ Y)=pz^*z^{*\H} - p I_n\approx pz^*z^{*\H}$, and thus for any $z\in\mathbb{C}_1^n$, the product $\mathbb{E}(A\circ Y)z \approx p(z^{*\H} z)z^*$ is approximately proportional to $z^*$ up to a global phase. Motivated by this fact, the power iteration repeatedly computes $(A\circ Y)z$ and applies normalization to each coordinate. The algorithm can be explicitly written as
\begin{equation}
z_j^{(t)} = \begin{cases}\frac{\sum_{k\in[n]\backslash\{j\}}A_{jk}Y_{jk}z_k^{(t-1)}}{\left|\sum_{k\in[n]\backslash\{j\}}A_{jk}Y_{jk}z_k^{(t-1)}\right|}, & \sum_{k\in[n]\backslash\{j\}}A_{jk}Y_{jk}z_k^{(t-1)}\neq 0, \\
1,  & \sum_{k\in[n]\backslash\{j\}}A_{jk}Y_{jk}z_k^{(t-1)}=0.
\end{cases}\label{eq:power}
\end{equation}
Let us shorthand the above formula as
\begin{equation}
z^{(t)} = f(z^{(t-1)}),\label{eqn:f}
\end{equation}
by introducing a map $f:\mathbb{C}_1^n\rightarrow\mathbb{C}_1^n$ such that the $j$th entry of $f(z^{(t-1)})$ is given by (\ref{eq:power}). The following lemma characterizes the evolution of the loss function (\ref{eq:loss}) through the map $f$.

\begin{lemma}\label{lem:key}
Assume $\sigma^2=o(np)$ and $\frac{np}{\log n}\rightarrow\infty$. Then, for any $\gamma=o(1)$, there exist some $\delta_1=o(1)$ and $\delta_2=o(1)$ such that
\begin{eqnarray*}
&& \mathbb{P}\left(\ell(f(z),z^*)\leq \delta_1 \ell(z,z^*) + (1+\delta_2)\frac{\sigma^2}{2p}\quad\text{for any }z\in \mathbb{C}_1^n\text{ such that }\ell(z,z^*)\leq \gamma n\right) \\
&\geq& 1-(2n)^{-1}-\exp\left(-\left(\frac{np}{\sigma^2}\right)^{1/4}\right).
\end{eqnarray*}
In particular, $\delta_1$ and $\delta_2$ can be chosen to satisfy $\delta_1=O\left(\left(\gamma^2+\frac{\log n+\sigma^2}{np}\right)^{1/4}\right)$ and $\delta_2=O\left(\sqrt{\frac{\log n+\sigma^2}{np}}\right)$.
\end{lemma}

The lemma shows that for any $z\in\mathbb{C}_1^n$ that has a nontrivial error, the vector $f(z)$ will have an error that is smaller by a multiplicative factor $\delta_1$ up to an additive term $(1+\delta_2)\frac{\sigma^2}{2p}$. Plugging $z=z^*$ into the inequality reveals that
$$\ell(f(z^*),z^*)\leq (1+\delta_2)\frac{\sigma^2}{2p}.$$
Thus, the additive term $(1+\delta_2)\frac{\sigma^2}{2p}$ can be understood as the oracle error given the knowledge of $z^*$. Indeed, we show in Theorem \ref{thm:lower} that $\frac{\sigma^2}{2p}$ is the minimax lower bound for phase synchronization under the loss function (\ref{eq:loss}).

The two conditions $\sigma^2=o(np)$ and $\frac{np}{\log n}\rightarrow\infty$ are essentially necessary for the result to hold. While $\sigma^2=o(np)$ makes sure that the error $\frac{\sigma^2}{2p}$ is of a nontrivial order, the condition $\frac{np}{\log n}\rightarrow\infty$ guarantees that the random graph is connected.
We can slightly relax both conditions to $np\geq C\sigma^2$ and $p\geq C\frac{\log n}{n}$ for some sufficiently large constant $C>0$ at the cost of replacing $\delta_1$ and $\delta_2$ in the result of Lemma \ref{lem:key} by some sufficiently small constants. However, vanishing $\delta_1$ and $\delta_2$ require that $\sigma^2$ to be of smaller order than $np$ and $p$ to be of greater order than $\frac{\log n}{n}$.

Lemma \ref{lem:key} implies that as long as $\ell(z^{(t-1)},z^*)\leq \gamma n$ for some $\gamma=o(1)$, the next step of power iteration (\ref{eq:power}) satisfies
\begin{equation}
\ell(z^{(t)},z^*) \leq \delta_1\ell(z^{(t-1)},z^*) + (1+\delta_2)\frac{\sigma^2}{2p}. \label{eq:inter-lem-key}
\end{equation}
The condition $\ell(z^{(t-1)},z^*)\leq \gamma n$ then implies $\ell(z^{(t)},z^*)\leq \delta_1\gamma n+(1+\delta_2)\frac{\sigma^2}{2p}$. Given that $\frac{\sigma^2}{2p}=o(n)$, we can always choose $\gamma=o(1)$ that satisfies $\frac{\sigma^2}{2p}=o(\gamma n)$. Therefore, $\ell(z^{(t)},z^*)\leq \gamma n$. Thus, a simple induction argument implies that (\ref{eq:inter-lem-key}) holds for all $t\geq 1$ as long as $\ell(z^{(0)},z^*)\leq \gamma n$. The one-step iteration bound (\ref{eq:inter-lem-key}) immediately implies the linear convergence
\begin{equation}
\ell(z^{(t)},z^*) \leq \delta_1^t\ell(z^{(0)},z^*) + \frac{1+\delta_2}{1-\delta_1}\frac{\sigma^2}{2p}, \label{eq:inter-lem-geo}
\end{equation}
for all $t\geq 1$.

\section{Optimality of Generalized Power Method and MLE}\label{sec:upper}

In this section, we show that both the GPM and the MLE achieve the optimal error $\frac{\sigma^2}{2p}$ by using the conclusion of Lemma \ref{lem:key}.  Theorem \ref{thm:power} and Theorem \ref{thm:MLE} together establish the upper bounds (\ref{eqn:intro_2}) for Theorem \ref{thm:intro}.

\subsection{Generalized Power Method}
The result of Lemma \ref{lem:key} implies that (\ref{eq:inter-lem-geo}) holds for all $t\geq 1$, as long as $\ell(z^{(0)},z^*)\leq \gamma n$ for some $\gamma=o(1)$. Since $\mathbb{E}(A\circ Y)$ is approximately a rank-one matrix, we can compute the leading eigenvector of $A\circ Y$ as the initialization of the power method. Let $\wh{u}\in\mathbb{C}^n$ be the leading eigenvector of $A\circ Y$, and we define $z^{(0)}$ according to
\begin{equation}
{z}_j^{(0)}=\begin{cases}\frac{\wh{u}_j}{|\wh{u}_j|}, & \wh{u}_j\neq 0, \\
1, & \wh{u}_j=0.
\end{cases}
\label{eq:spec-ini}
\end{equation}
It is easy to check that $z^{(0)}\in\mathbb{C}_1^n$, and thus  can be applied by the power iteration (\ref{eq:power}). This leads to the GPM for phase synchronization, which is formally presented as Algorithm \ref{alg}.

\begin{algorithm}[H]
\SetAlgoLined
\KwIn{The data $A\circ Y$ and the number of iterations $t_{\max}$.}
\KwOut{The estimator $\wh{z}=z^{t_{\max}}$.}
 \nl Obtain $z^{(0)}$ from the leading eigenvector of $A\circ Y$ by (\ref{eq:spec-ini})\;
 \nl \For{$t=1,\cdots,t_{\max}$}{
    $$z^{(t)} = f(z^{(t-1)}),$$
    where $f(\cdot)$ is defined in (\ref{eq:power}).
 }
\caption{Generalized Power Method for Phase Synchronization \label{alg}}
\end{algorithm}

The error bound of $z^{(0)}$ is given by the following lemma.

\begin{lemma}\label{lem:spec-ini}
Assume $\frac{np}{\log n}\rightarrow\infty$. Then, there exists some constant $C>0$ such that
$$\ell(z^{(0)},z^*)\leq C\frac{\sigma^2+1}{p},$$
with probability at least $1-(2n)^{-1}$.
\end{lemma}

Under the condition that $\sigma^2=o(np)$ and $\frac{np}{\log n}\rightarrow\infty$, the error rate of Lemma \ref{lem:spec-ini} satisfies $\frac{\sigma^2+1}{p}=o(n)$ so that $\ell(z^{(0)},z^*)\leq\gamma n$ holds for some $\gamma=o(1)$. By Lemma \ref{lem:key} and its implication (\ref{eq:inter-lem-geo}), we directly obtain the following result.

\begin{thm}\label{thm:power}
Assume $\sigma^2=o(np)$ and $\frac{np}{\log n}\rightarrow\infty$. Then, there exists some $\delta=o(1)$, such that the GPM
(i.e, Algorithm \ref{alg})
satisfies
$$\ell(z^{(t)},z^*) \leq (1+\delta)\frac{\sigma^2}{2p},$$
for all $t\geq \log\left(\frac{1}{\sigma^2}\right)$ with probability at least $1-n^{-1}-\exp\left(-\left(\frac{np}{\sigma^2}\right)^{1/4}\right)$. In particular, $\delta$ can be chosen to satisfy $\delta=O\left(\left(\frac{\log n+\sigma^2}{np}\right)^{1/4}\right)$.
\end{thm}

We remark that the number of iterations $\log\left(\frac{1}{\sigma^2}\right)$ required by the theorem can be improved in some special cases. For example, when $p=1$, the error bound of Lemma \ref{lem:spec-ini} can be improved to $\ell(z^{(0)},z^*)\leq C\sigma^2$ by a matrix perturbation analysis \citep{boumal2016nonconvex}. Then, (\ref{eq:inter-lem-key}) implies that $\ell(z^{(1)},z^*)\leq (1+o(1))\frac{\sigma^2}{2}$. In other words, when the graph is fully connected, a one-step refinement of power iteration is sufficient to achieve the optimal error.

\subsection{Maximum Likelihood Estimator}
Next, we discuss how the result of Lemma \ref{lem:key} also implies the optimality of the MLE. According to the data generating process, the MLE is given by
\begin{equation}
\wh{z}=\argmin_{z\in\mathbb{C}_1^n}\sum_{1\leq j< k\leq n}A_{jk}|Y_{jk}-z_j\bar{z}_k|^2, \label{eq:MLE}
\end{equation}
which is equivalent to $\argmax_{z\in\mathbb{C}_1^n} z^{\H} (A\circ Y)z$.
By the definition of $\wh{z}$, its $j$th entry must satisfy
$$\wh{z}_j=\argmin_{z_j\in\mathbb{C}_1}\sum_{k\in[n]\backslash\{j\}}A_{jk}|Y_{jk}-z_j\bar{\wh{z}}_k|^2=\frac{\sum_{k\in[n]\backslash\{j\}}A_{jk}Y_{jk}\wh{z}_k}{\left|\sum_{k\in[n]\backslash\{j\}}A_{jk}Y_{jk}\wh{z}_k\right|}.$$
In other words, we have
$$\wh{z}=f(\wh{z}),$$
and the MLE is a fixed point of the power iteration. As long as we can establish a crude bound $\ell(\wh{z},z^*)\leq \gamma n$ for some $\gamma=o(1)$, Lemma \ref{lem:key} automatically leads to the optimal error.

\begin{lemma}\label{lem:MLE-ini}
Assume $\frac{np}{\log n}\rightarrow\infty$. Then, there exists some constant $C>0$ such that the MLE has error bound
$$\ell(\wh{z},z^*)\leq C\frac{\sigma^2+1}{p},$$
with probability at least $1-(2n)^{-1}$.
\end{lemma}

Again, under the condition that $\sigma^2=o(np)$ and $\frac{np}{\log n}\rightarrow\infty$, we have $\ell(\wh{z},z^*)\leq \gamma n$ for some $\gamma=o(1)$. Lemma \ref{lem:key} and the fact $\wh{z}=f(\wh{z})$ implies that
$$\ell(\wh{z},z^*)\leq \delta_1\ell(\wh{z},z^*)+(1+\delta_2)\frac{\sigma^2}{2p}.$$
After rearrangement, we obtain the bound $\ell(\wh{z},z^*)\leq \frac{1+\delta_2}{1-\delta_1}\frac{\sigma^2}{p}$. The result is summarized into the following theorem.

\begin{thm}\label{thm:MLE}
Assume $\sigma^2=o(np)$ and $\frac{np}{\log n}\rightarrow\infty$. Then, there exists some $\delta=o(1)$, such that the MLE (\ref{eq:MLE}) satisfies
$$\ell(\wh{z},z^*) \leq (1+\delta)\frac{\sigma^2}{2p},$$
with probability at least $1-n^{-1}-\exp\left(-\left(\frac{np}{\sigma^2}\right)^{1/4}\right)$. In particular, $\delta$ can be chosen to satisfy $\delta=O\left(\left(\frac{\log n+\sigma^2}{np}\right)^{1/4}\right)$.
\end{thm}

To close this section, we briefly discuss the implication of Lemma \ref{lem:key} on semi-definite programming (SDP) when the graph is fully connected. In other words, we only consider $p=1$. This is by leveraging a recent result on the connection between SDP and MLE by \cite{zhong2018near}.
Recall the definition of SDP in (\ref{eq:SDP}).
The following result is a special case of Theorem 5 of \cite{zhong2018near}.

\begin{thm}[\cite{zhong2018near}]\label{thm:MLE-SDP}
Assume $\sigma^2=O\left(\frac{n}{\log n}\right)$. Then, with probability at least $1-n^{-1}$, the SDP (\ref{eq:SDP}) admits a unique solution $\wh{z}\wh{z}^{\H}$, where $\wh{z}$ is a global optimum of (\ref{eq:MLE}) with all $A_{jk}=1$.
\end{thm}

This equivalence of SDP and MLE immediately implies the following result.

\begin{corollary}\label{cor:SDP}
Assume $\sigma^2=O\left(\frac{n}{\log n}\right)$ and $p=1$. Then, there exists some $\delta=o(1)$, such that the unique solution of the SDP can be written as $\wh{z}\wh{z}^{\H}$, where $\wh{z}$ satisfies
$$\ell(\wh{z},z^*) \leq (1+\delta)\frac{\sigma^2}{2},$$
with probability at least $1-2n^{-1}-\exp\left(-\left(\frac{n}{\sigma^2}\right)^{1/4}\right)$. In particular, $\delta$ can be chosen to satisfy $\delta=O\left(\left(\frac{\log n+\sigma^2}{n}\right)^{1/4}\right)$.
\end{corollary}

Corollary \ref{cor:SDP} requires a stronger condition $\sigma^2=O\left(\frac{n}{\log n}\right)$ than what is needed by Theorem \ref{thm:power} and Theorem \ref{thm:MLE}  when $p=1$. This condition is needed for the equivalence between SDP and MLE, which is established via an $\ell_{\infty}$ norm argument in \citep{zhong2018near} and is thus very unlikely to be weakened. Whether the equivalence between SDP and MLE continues to hold when $p<1$ is a less clear issue in the literature. In general, the solution to the SDP (\ref{eq:SDP}) is not necessarily a rank-one matrix. Obtaining an estimator in $\mathbb{C}_1^n$ requires a post-processing step such as the rounding method suggested by \cite{javanmard2016phase}, the consequence of which is unclear to us.

\section{Minimax Lower Bound}\label{sec:lower}

We study the minimax lower bound of phase synchronization in this section. The minimax risk is given by
\begin{equation}
\inf_{\wh{z}\in\mathbb{C}_1^n}\sup_{z\in\mathbb{C}_1^n}\mathbb{E}_z\ell(\wh{z},z)=\inf_{\wh{z}\in\mathbb{C}_1^n}\sup_{z\in\mathbb{C}_1^n}\mathbb{E}_z\left[\min_{a\in\mathbb{C}_1}\sum_{j=1}^n|\wh{z}_j a-z_j|^2\right], \label{eq:lower-1st}
\end{equation}
where $\mathbb{E}_z$ is the expectation of $(A\circ Y,A)$ under the model $A_{jk}Y_{jk}=A_{jk}z_j\bar{z}_k+\sigma A_{jk}W_{jk}$ and $A_{jk}\sim\text{Bernoulli}(p)$. The main difficulty of analyzing (\ref{eq:lower-1st}) is that the loss function $\ell(\wh{z},z)$ is not separable in $j\in[n]$ because of the minimization over $\theta\in\mathbb{R}$. This difficulty can be tackled with the following inequality
\begin{equation}
\ell(z,z^*)\geq \frac{1}{2n}\fnorm{\wh{z}\wh{z}^{\H}-zz^{\H}}^2, \label{eq:clever}
\end{equation}
which is proved in Lemma \ref{lem:loss-equiv}. Since the right hand side of (\ref{eq:clever}) can be written as $\sum_{1\leq j\neq k\leq n}|\wh{z}_j\bar{\wh{z}}_k-z_j\bar{z}_k|^2$, we can lower bound (\ref{eq:lower-1st}) by
\begin{eqnarray}
\nonumber && \inf_{\wh{z}\in\mathbb{C}_1^n}\sup_{z\in\mathbb{C}_1^n}\mathbb{E}_z\ell(\wh{z},z) \\
\nonumber &\geq& \frac{1}{2n}\inf_{\wh{z}\in\mathbb{C}_1^n}\sup_{z\in\mathbb{C}_1^n}\mathbb{E}_z\fnorm{\wh{z}\wh{z}^{\H}-zz^{\H}}^2 \\
\nonumber &\geq& \frac{1}{2n}\inf_{\wh{z}}\sum_{1\leq j\neq k\leq n}\left(\int\prod_{l=1}^n\pi(z_l)\right)\mathbb{E}_z|\wh{z}_j\bar{\wh{z}}_k-z_j\bar{z}_k|^2\diff z \\
\label{eq:lamp} &\geq& \frac{1}{2n}\sum_{1\leq j\neq k\leq n}\mathbb{E}_{z_{-(j,k)}\sim\pi}\inf_{\wh{T}}\int \int \pi(z_j)\pi(z_k)\mathbb{E}_z|\wh{T}-z_j\bar{z}_k|^2\diff z_j\diff z_k,
\end{eqnarray}
where $\pi$ is some density function supported on $\mathbb{C}_1$ to be specified later, and the notation $\mathbb{E}_{z_{-(j,k)}}$ means expectation over $z$ except for its $j$th and $k$th entries with respect to the distribution $\pi$. Therefore, it is sufficient to lower bound $\inf_{\wh{T}}\int\int \pi(z_j)\pi(z_k)\mathbb{E}_z|\wh{T}-z_j\bar{z}_k|^2\diff z_j\diff z_k$ for all $j\neq k$. This corresponds to the problem of estimating $z_j\bar{z}_k$ with other entries of $z$ assumed to be known.

By symmetry, we consider the problem with $j=1$ and $k=2$. Given the knowledge of $z_3,\cdots,z_n$, we can decompose the likelihood function as
\begin{eqnarray*}
p(A\circ Y, A) &=& p(A)p(A_{12}Y_{12}|A)\left(\prod_{j=3}^np(A_{1j}Y_{1j}|A)p(A_{2j}Y_{2j}|A)\right) \\
&& \times \prod_{3\leq j<k\leq n}p(A_{jk}Y_{jk}|A).
\end{eqnarray*}
Since $p(A)\prod_{3\leq j<k\leq n}p(A_{jk}Y_{jk}|A)$ is independent of $z_1\bar{z}_2$,
the sufficient statistics for estimating $z_1\bar{z}_2$ is $A_{12}Y_{12}, A_{13}Y_{13},\cdots, A_{1n}Y_{1n}$ and $A_{23}Y_{23},\cdots, A_{2n}Y_{2n}$. Now we restrict $z_j\sim \pi$ by requiring $z_j=a+\sqrt{1-a^2}i$ with $a\sim f$ for some density $f$ supported on the unit interval $[0,1]$. Then, we can represent $z_1$ and $z_2$ by
$$z_1=a+\sqrt{1-a^2}i\quad\text{and}\quad z_2=c+\sqrt{1-c^2}i.$$
Define
\begin{eqnarray*}
T(a,c) &=& ac+\sqrt{1-a^2}\sqrt{1-c^2}, \\
S(a,c) &=& \sqrt{1-a^2}c-a\sqrt{1-c^2}.
\end{eqnarray*}
With this notation, we have
$$z_1\bar{z}_2=T(a,c)+S(a,c)i.$$
The likelihood $p(A_{12}Y_{12}|A)$ corresponds to the distribution
\begin{equation}
\n\left(A_{12}\begin{pmatrix}
T(a,c) \\
S(a,c)
\end{pmatrix}, \frac{1}{2}A_{12}\sigma^2I_2\right). \label{eq:cond1}
\end{equation}
For $j=3,\cdots,n$, we have $Y_{1j}=z_1\bar{z}_j+\sigma W_{1j}$. Since $z_j$ is known for $j=3,\cdots,n$, observing $Y_{1j}$ is equivalent to observing $Y_{1j}z_j=z_1+\sigma W_{1j}z_j$, and we still have $W_{1j}z_j\sim\cn(0,1)$ by the property of complex Gaussian distribution. This argument implies that the likelihood $\prod_{j=3}^np(A_{1j}Y_{1j}|A)$ is proportional to the density function of
\begin{equation}
\n\left(\sum_{j=3}^nA_{1j}\begin{pmatrix}
a \\
\sqrt{1-a^2}
\end{pmatrix},\frac{\sigma^2\sum_{j=3}^nA_{1j}}{2}I_2\right). \label{eq:cond2}
\end{equation}
Similarly, $\prod_{j=3}^np(A_{2j}Y_{2j}|A)$ is proportional to the density function of
\begin{equation}
\n\left(\sum_{j=3}^nA_{2j}\begin{pmatrix}
c \\
\sqrt{1-c^2}
\end{pmatrix},\frac{\sigma^2\sum_{j=3}^nA_{2j}}{2}I_2\right). \label{eq:cond3}
\end{equation}
Therefore, we have the identity
\begin{eqnarray}
\nonumber && \inf_{\wh{T}}\int \int \pi(z_1)\pi(z_2)\mathbb{E}_z|\wh{T}-z_1\bar{z}_2|^2\diff z_1\diff z_2 \\
\label{eq:mse-bayes} &=& \inf_{\wh{T},\wh{S}}\int \int f(a)f(c)\mathbb{E}_{(a,c)}\left[(\wh{T}-T(a,c))^2+(\wh{S}-S(a,c))^2\right]\diff a \diff c,
\end{eqnarray}
where the expectation $\mathbb{E}_{(a,c)}$ is under the product of the conditional distributions (\ref{eq:cond1}), (\ref{eq:cond2}) and (\ref{eq:cond3}) with marginal given by $A_{jk}\sim\text{Bernoulli}(p)$. The quantity (\ref{eq:mse-bayes}) can be lower bounded by van Trees' inequality \citep{gill1995applications}, and the result is given by the following lemma.

\begin{lemma}\label{lem:lower-van-trees}
Assume $\sigma^2=o(np)$. Then, there exists some nice density function $f$ supported on the unit interval and some $\delta=o(1)$ such that
$$\inf_{\wh{T},\wh{S}}\int \int f(a)f(c)\mathbb{E}_{(a,c)}\left[(\wh{T}-T(a,c))^2+(\wh{S}-S(a,c))^2\right]\diff a \diff c \geq (1-\delta)\frac{\sigma^2}{np}.$$
In particular, $\delta$ can be chosen to satisfy $\delta=O\left(\frac{\sigma^2}{np}+\frac{1}{n}\right)$.
\end{lemma}

Plugging the result of Lemma \ref{lem:lower-van-trees} into (\ref{eq:lamp}), we have
immediately the following theorem, which establishes the lower bound (\ref{eqn:intro_1}) for Theorem \ref{thm:intro}.

\begin{thm}\label{thm:lower}
Assume $\sigma^2=o(np)$. Then, there exists some $\delta=o(1)$ such that
$$\inf_{\wh{z}\in\mathbb{C}_1^n}\sup_{z\in\mathbb{C}_1^n}\mathbb{E}_z\ell(\wh{z},z)\geq (1-\delta)\frac{\sigma^2}{2p}.$$
In particular, $\delta$ can be chosen to satisfy $\delta=O\left(\frac{\sigma^2}{np}+\frac{1}{n}\right)$.
\end{thm}


\section{Proofs}\label{sec:proof}

This section presents the proofs of all technical results in the paper. We first list some auxiliary lemmas in Section \ref{sec:tech}. The key lemma that leads to various upper bound results (Lemma \ref{lem:key}) is proved in Section \ref{sec:pf-key}. Then, we prove Lemma \ref{lem:lower-van-trees} in Section \ref{sec:pf-low}. Finally, the proofs of Lemma \ref{lem:spec-ini} and Lemma \ref{lem:MLE-ini} are given in Section \ref{sec:pf-ini}.

\subsection{Some Auxiliary Lemmas}\label{sec:tech}

\begin{lemma}\label{lem:ER-graph}
Assume $\frac{np}{\log n}\rightarrow\infty$. Then, there exists a constant $C>0$, such that
$$\max_{j\in[n]}\left(\sum_{k\in[n]\backslash\{j\}}(A_{jk}-p)\right)^2\leq Cnp\log n,$$
and
$$\opnorm{A-\mathbb{E}A}\leq C\sqrt{np},$$
with probability at least $1-n^{-10}$.
\end{lemma}
\begin{proof}
The first result is a direct application of union bound and Bernstein's inequality. The second result is Theorem 5.2 of \cite{lei2015consistency}. 
\end{proof}

\begin{lemma}\label{lem:bandeira}
Assume $\frac{np}{\log n}\rightarrow\infty$. Then, there exists a constant $C>0$, such that
$$\opnorm{A\circ W}\leq C\sqrt{np},$$
with probability at least $1-n^{-10}$.
\end{lemma}
\begin{proof}
We use $\mathbb{P}_A$ for the conditional probability $\mathbb{P}(\cdot|A)$. Define the event
$$\mathcal{A}=\left\{\max_{j\in[p]}\sum_{k\in[n]\backslash\{j\}}A_{jk}\leq 2np\right\}.$$
Under the assumption $\frac{np}{\log n}\rightarrow\infty$, we have $\mathbb{P}(\mathcal{A}^c)\leq n^{-11}$ by Bernstein's inequality and a union bound argument. By Corollary 3.11 of \cite{bandeira2016sharp}, we have
$$\sup_{A\in\mathcal{A}}\mathbb{P}_A\left(\opnorm{A\circ \re(W)} > C_1\sqrt{np} + t\right)\leq e^{-t^2/2},$$
for some constant $C_1>0$. This implies that $\sup_{A\in\mathcal{A}}\mathbb{P}_A\left(\opnorm{A\circ \re(W)} > C_2\sqrt{np}\right)\leq n^{-11}$ for some constant $C_2>0$. Thus, we have
$$\mathbb{P}\left(\opnorm{A\circ \re(W)} > C_2\sqrt{np}\right)\leq \mathbb{P}(\mathcal{A}^c)+\sup_{A\in\mathcal{A}}\mathbb{P}_A\left(\opnorm{A\circ \re(W)} > C_2\sqrt{np}\right)\leq 2n^{-11}.$$
The same high probability bound also holds for $\opnorm{A\circ \im(W)}$. Finally, the desired conclusion is implied by $\opnorm{A\circ W}\lesssim \opnorm{A\circ \im(W)}+\opnorm{A\circ \re(W)}$.
\end{proof}

\begin{lemma}[Lemma 13 of \cite{gao2016optimal}]\label{lem:gaomalu}
Consider independent random variables $X_j\sim \n(0,1)$ and $E_j\sim\text{Bernoulli}(p)$. Then,
$$\mathbb{P}\left(\left|\sum_{j=1}^nX_jE_j/p\right|>t\right)\leq 2\exp\left(-\min\left(\frac{pt^2}{16n},\frac{pt}{2}\right)\right),$$
for any $t>0$.
\end{lemma}

\begin{lemma}\label{lem:loss-equiv}
For any $z,z^*\in\mathbb{C}_1^n$, we have
$$n\ell(z,z^*)\leq \fnorm{zz^{\H}-z^*z^{*\H}}^2\leq 2n\ell(z,z^*).$$
\end{lemma}
\begin{proof}
The definition of $\ell(z,z^*)$ implies
$$\ell(z,z^*)=2n-\sup_{|a|=1}\left({z}^{\H}z^*a+z^{*\H}z\bar{a}\right)=2\left(n-|z^{\H}z^*|\right).$$
By direct calculation,
$$\fnorm{zz^{\H}-z^*z^{*\H}}^2=2\left(n^2-|z^{\H}z^*|^2\right).$$
We thus obtain the relation
$$\fnorm{zz^{\H}-z^*z^{*\H}}^2=\ell(z,z^*)\left(n+|z^{\H}z^*|\right)=\ell(z,z^*)\left(2n-\frac{\ell(z,z^*)}{2}\right),$$
which immediately implies the conclusion.
\end{proof}

\begin{lemma}\label{lem:middle}
For any $x\in\mathbb{C}$ such that $\re(x)>0$, $\left|\frac{x}{|x|}-1\right|\leq \left|\frac{\im(x)}{\re(x)}\right|$.
\end{lemma}
\begin{proof}
Let $x=a+bi$, and then
\begin{eqnarray*}
\left|\frac{x}{|x|}-1\right| &=& \frac{|x-|x||}{|x|} \\
&=& \frac{\left|a+bi-\sqrt{a^2+b^2}\right|}{|x|} \\
&=& \frac{\sqrt{(a-\sqrt{a^2+b^2})^2+b^2}}{\sqrt{a^2+b^2}} \\
&=& \sqrt\frac{2b^2}{a^2+b^2+a\sqrt{a^2+b^2}} \\
&\leq& \left|\frac{b}{a}\right| = \left|\frac{\im(x)}{\re(x)}\right|.
\end{eqnarray*}
The proof is complete.
\end{proof}

\begin{lemma}\label{lem:trivial}
For any $x\in\mathbb{C}\backslash\{0\}$ and any $y\in\mathbb{C}_1^n$, $\left|\frac{x}{|x|}-y\right|\leq 2|x-y|$.
\end{lemma}
\begin{proof}
We have $\left|\frac{x}{|x|}-y\right|\leq \left|\frac{x}{|x|}-x\right|+|x-y|= ||x|-1|+|x-y|=||x|-|y||+|x-y|\leq 2|x-y|$.
\end{proof}

\subsection{Proof of Lemma \ref{lem:key}}\label{sec:pf-key}

We organize the proof into four steps. We first list a few high-probability events in Step 1. These events are assumed to be true in later steps. Step 2 provides an error decomposition of $\ell(f(z),z^*)$, and then each error term in the decomposition will be analyzed and bounded in Step 3. Finally, we combine the bounds and derive the desired result in Step 4.
\paragraph{Step 1: Some high-probability events.} By Lemma \ref{lem:ER-graph} and Lemma \ref{lem:bandeira}, we know that 
\begin{eqnarray}
\label{eq:hollow1} \min_{j\in[n]}\sum_{k\in[n]\backslash\{j\}}A_{jk} &\geq& (n-1)p - C\sqrt{np\log n}, \\
\label{eq:hollow2}\max_{j\in[n]}\sum_{k\in[n]\backslash\{j\}}A_{jk} &\leq& (n-1)p + C\sqrt{np\log n}, \\
\label{eq:hollow3}\opnorm{A-\mathbb{E}A} &\leq& C\sqrt{np}, \\
\label{eq:hollow4}\opnorm{A\circ W} &\leq& C\sqrt{np},
\end{eqnarray}
all hold with probability at least $1-n^{-9}$ for some constant $C>0$. In addition to (\ref{eq:hollow1})-(\ref{eq:hollow4}), we need two more high-probability inequalities. For $\rho$ that satisfies $\rho\rightarrow 0$ and $\frac{\rho^2 np}{\sigma^2}\rightarrow\infty$, we want to bound the random variable $\sum_{j=1}^n\mathbb{I}\left\{\frac{2\sigma}{np}\left|\sum_{k\in[n]\backslash\{j\}}A_{jk}W_{jk}z_k^*\right|>\rho\right\}$. The existence of such $\rho$ is guaranteed by the condition $\sigma^2=o(np)$, and the specific choice will be given later. We first bound its expectation by Lemma \ref{lem:gaomalu},
\begin{eqnarray*}
\sum_{j=1}^n\mathbb{P}\left\{\frac{2\sigma}{np}\left|\sum_{k\in[n]\backslash\{j\}}A_{jk}W_{jk}z_k^*\right|>\rho\right\} &\leq& \sum_{j=1}^n\mathbb{P}\left\{\frac{2\sigma}{np}\left|\sum_{k\in[n]\backslash\{j\}}A_{jk}\re(W_{jk}z_k^*)\right|>\frac{\rho}{2}\right\} \\
&& + \sum_{j=1}^n\mathbb{P}\left\{\frac{2\sigma}{np}\left|\sum_{k\in[n]\backslash\{j\}}A_{jk}\im(W_{jk}z_k^*)\right|>\frac{\rho}{2}\right\} \\
&\leq& 4n\exp\left(-\frac{\rho^2 np}{256\sigma^2}\right) + 4n\exp\left(-\frac{\rho np}{8\sigma}\right).
\end{eqnarray*}
By Markov inequality, we have
\begin{equation}
\sum_{j=1}^n\mathbb{I}\left\{\frac{2\sigma}{np}\left|\sum_{k\in[n]\backslash\{j\}}A_{jk}W_{jk}z_k^*\right|>\rho\right\} \leq \frac{4\sigma^2}{\rho^2p} \exp\left(-\frac{1}{16}\sqrt{\frac{\rho^2 np}{\sigma^2}}\right), \label{eq:hollow5}
\end{equation}
with probability at least
\begin{eqnarray*}
&& 1-\frac{\rho^2pn}{\sigma^2}\left(\exp\left(-\frac{\rho^2 np}{256\sigma^2}+\frac{1}{16}\sqrt{\frac{\rho^2 np}{\sigma^2}}\right) + \exp\left(-\frac{\rho np}{8\sigma}+\frac{1}{16}\sqrt{\frac{\rho^2 np}{\sigma^2}}\right)\right) \\
&\geq& 1-\frac{2\rho^2pn}{\sigma^2}\exp\left(-\frac{1}{16}\sqrt{\frac{\rho^2 np}{\sigma^2}}\right) \\
&\geq& 1-\exp\left(-\frac{1}{32}\sqrt{\frac{\rho^2 np}{\sigma^2}}\right).
\end{eqnarray*}
The second high-probability bound we need is for the random variable $\sum_{j=1}^n\left|\sum_{k\in[n]\backslash\{j\}}A_{jk}\im(W_{jk}z_k^*\bar{z}_j^*)\right|^2$. We first find its expectation. By direct calculation, we have
$$\sum_{j=1}^n\mathbb{E}\left|\sum_{k\in[n]\backslash\{j\}}A_{jk}\im(W_{jk}z_k^*\bar{z}_j^*)\right|^2=\frac{n(n-1)p}{2}.$$
To study its variance, we introduce the notation $\epsilon_{jk}=A_{jk}\im(W_{jk}z_k^*\bar{z}_j^*)$ and $\epsilon_j=\sum_{k\in[n]\backslash\{j\}}\epsilon_{jk}$. The random variable $\epsilon_{jk}$ satisfies the property
$$-\epsilon_{jk}=-A_{jk}\im(W_{jk}z_k^*\bar{z}_j^*)=A_{kj}\im(\bar{W}_{jk}\bar{z}_k^*z_j^*)=A_{jk}\im(W_{kj}z_j^*\bar{z}_k^*)=\epsilon_{kj}.$$
With the new notation, we have
\begin{equation}
\Var\left(\sum_{j=1}^n\epsilon_j^2\right)\leq \sum_{j=1}^n\left(\mathbb{E}\epsilon_j^4-(\mathbb{E}\epsilon_j^2)\right)+\sum_{1\leq j\neq l\leq n}\left(\mathbb{E}\epsilon_j^2\epsilon_l^2 - \mathbb{E}\epsilon_j^2\mathbb{E}\epsilon_l^2\right).\label{eq:var-bound}
\end{equation}
For any $j\neq l$, we have
\begin{eqnarray*}
\mathbb{E}\epsilon_j^2\epsilon_l^2 &=& \mathbb{E}\left(\sum_{k\in[n]\backslash\{j\}}\epsilon_{jk}\right)^2\left(\sum_{k\in[n]\backslash\{l\}}\epsilon_{lk}\right)^2 \\
&=& \sum_{k_1\in[n]\backslash\{j\}}\sum_{k_2\in[n]\backslash\{j\}}\mathbb{E}\epsilon_{jk_1}^2\epsilon_{lk_2}^2.
\end{eqnarray*}
Observe that $\mathbb{E}\epsilon_{jk_1}^2\epsilon_{lk_2}^2$ is either $\frac{3p}{4}$ or $\frac{p^2}{4}$, depending on whether or not $(j,k_1)$ and $(l,k_2)$ correspond to the same edge. Therefore
$$\mathbb{E}\epsilon_j^2\epsilon_l^2 =\frac{(n^2-2n)p^2}{4}+\frac{3p}{4},$$
for any $j\neq l$. We also have for any $j$,
\begin{eqnarray*}
\mathbb{E}\epsilon_j^4 &=& \mathbb{E}\left(\sum_{k\in[n]\backslash\{j\}}\epsilon_{jk}\right)^4 \\
&=& \sum_{k\neq [n]\backslash\{j\}}\mathbb{E}\epsilon_{jk}^4 + \sum_{k,l\in[n]\backslash\{j\}:k\neq l}\mathbb{E}\epsilon_{jk}^2\mathbb{E}\epsilon_{lk}^2 \\
&=& \frac{3(n-1)p}{4} + \frac{(n^2-3n+2)p^2}{4}.
\end{eqnarray*}
We plug the above results into the bound (\ref{eq:var-bound}), and we have
$$\Var\left(\sum_{j=1}^n\epsilon_j^2\right)\leq \frac{3n(n-1)p}{4} + \frac{3n(n-1)p}{4}\leq \frac{3n^2p}{2}.$$
Therefore, by Chebyshev inequality, we can conclude that with probability at least $1-(6n)^{-1}$,
\begin{equation}
\sum_{j=1}^n\left|\sum_{k\in[n]\backslash\{j\}}A_{jk}\im(W_{jk}z_k^*\bar{z}_j^*)\right|^2\leq \frac{n^2p}{2}\left(1+\frac{6}{\sqrt{np}}\right). \label{eq:hollow6}
\end{equation}
Using the same analysis, we also have
\begin{equation}
\sum_{j=1}^n\left|\sum_{k\in[n]\backslash\{j\}}A_{jk}\re(W_{jk}z_k^*\bar{z}_j^*)\right|^2\leq \frac{n^2p}{2}\left(1+\frac{6}{\sqrt{np}}\right), \label{eq:hollow7}
\end{equation}
with probability at least $1-(6n)^{-1}$. Finally, we conclude that the events (\ref{eq:hollow1}), (\ref{eq:hollow2}), (\ref{eq:hollow3}), (\ref{eq:hollow4}), (\ref{eq:hollow5}), (\ref{eq:hollow6}) and (\ref{eq:hollow7}) hold simultaneously with probability at least $1-(2n)^{-1}-\exp\left(-\frac{1}{32}\sqrt{\frac{\rho^2 np}{\sigma^2}}\right)$.

\paragraph{Step 2: Error decomposition.}

For any $z\in\mathbb{C}_1^n$ such that $\ell(z,z^*)\leq \gamma n$, we can write $\wh{z}=f(z)$ with each coordinate $\wh{z}_j=\wt{z}_j/|\wt{z}_j|$, where
$$\wt{z}_j=\frac{\sum_{k\in[n]\backslash\{j\}}A_{jk}Y_{jk}z_k}{\sum_{k\in[n]\backslash\{j\}}A_{jk}}.$$
The condition $\ell(z,z^*)\leq \gamma n$ implies there exists some $b\in\mathbb{C}_1$ such that $\|z-z^*b\|^2\leq \gamma n$. By direct calculation, we can write
\begin{eqnarray*}
\wt{z}_j\bar{z}_j^*\bar{b} &=& \frac{\sum_{k\in[n]\backslash\{j\}}A_{jk} z_j^* \bar{z}_k^* z_k}{\sum_{k\in[n]\backslash\{j\}}A_{jk}} \bar{z}_j^* \bar{b} + \frac{\sum_{k\in[n]\backslash\{j\}}A_{jk} W_{jk}z_k}{\sum_{k\in[n]\backslash\{j\}}A_{jk}} \bar{z}_j^* \bar{b}\\
&=& 1 + \frac{\sum_{k\in[n]\backslash\{j\}}A_{jk}\bar{z}_k^*\bar{b}(z_k-{z}_k^*b)}{\sum_{k\in[n]\backslash\{j\}}A_{jk}} + \frac{\sigma \sum_{k\in[n]\backslash\{j\}}A_{jk}W_{jk}(z_k-z_k^*b)\bar{z}_j^*\bar{b}}{\sum_{k\in[n]\backslash\{j\}}A_{jk}} \\
&& + \frac{\sigma \sum_{k\in[n]\backslash\{j\}}A_{jk}W_{jk}z_k^*\bar{z}_j^*}{\sum_{k\in[n]\backslash\{j\}}A_{jk}} \\
&=& 1 + \frac{1}{n-1}\sum_{k=1}^n\bar{z}_k^*\bar{b}(z_k-z_k^*b) - \frac{1}{n-1}\bar{z}_j^*\bar{b}(z_j-z_j^*b) \\
&& + \left(\frac{\sum_{k\in[n]\backslash\{j\}}A_{jk}\bar{z}_k^*\bar{b}(z_k-{z}_k^*b)}{\sum_{k\in[n]\backslash\{j\}}A_{jk}}-\frac{1}{n-1}\sum_{k\in[n]\backslash\{j\}}^n\bar{z}_k^*\bar{b}(z_k-z_k^*b)\right) \\
&& + \frac{\sigma \sum_{k\in[n]\backslash\{j\}}A_{jk}W_{jk}(z_k-z_k^*b)\bar{z}_j^*\bar{b}}{\sum_{k\in[n]\backslash\{j\}}A_{jk}} + \frac{\sigma \sum_{k\in[n]\backslash\{j\}}A_{jk}W_{jk}z_k^*\bar{z}_j^*}{\sum_{k\in[n]\backslash\{j\}}A_{jk}}.
\end{eqnarray*}
Now we define $a_0=1 + \frac{1}{n-1}\sum_{k=1}^n\bar{z}_k^*\bar{b}(z_k-z_k^*b)$ and $a=a_0/|a_0|$, and we have
$$\wt{z}_j\bar{z}_j^*\bar{a}\bar{b}=|a_0| - \frac{1}{n-1}\bar{z}_j^*\bar{a}\bar{b}(z_j-z_j^*b) + F_j+G_j+H_j,$$
where
\begin{eqnarray*}
F_j &=& \frac{\sum_{k\in[n]\backslash\{j\}}A_{jk}\bar{z}_k^*\bar{a}\bar{b}(z_k-{z}_k^*b)}{\sum_{k\in[n]\backslash\{j\}}A_{jk}}-\frac{1}{n-1}\sum_{k\in[n]\backslash\{j\}}^n\bar{z}_k^*\bar{a}\bar{b}(z_k-z_k^*b), \\
G_j &=& \frac{\sigma \sum_{k\in[n]\backslash\{j\}}A_{jk}W_{jk}(z_k-z_k^*b)\bar{z}_j^*\bar{a}\bar{b}}{\sum_{k\in[n]\backslash\{j\}}A_{jk}}, \\
H_j &=& \frac{\sigma \bar{a}\sum_{k\in[n]\backslash\{j\}}A_{jk}W_{jk}z_k^*\bar{z}_j^*}{\sum_{k\in[n]\backslash\{j\}}A_{jk}}.
\end{eqnarray*}
By Lemma \ref{lem:middle}, we have the bound
\begin{equation}
|\wh{z}_j-z_j^*ab|^2=|\wh{z}_j\bar{z}_j^*\bar{a}\bar{b}-1|^2 \leq \left|\frac{\im{(\wt{z}_j\bar{z}_j^*\bar{a}\bar{b})}}{\re{(\wt{z}_j\bar{z}_j^*\bar{a}\bar{b})}}\right|^2, \label{eq:ratio-im-re}
\end{equation}
whenever $\re{(\wt{z}_j\bar{z}_j^*\bar{a}\bar{b})}>0$ holds.
By Cauchy-Schwartz inequality, we have
\begin{equation}
|a_0|\geq 1-\frac{1}{n-1}\|z^*\|\|z-bz^*\|\geq 1-2\sqrt{\gamma}, \label{eq:ori-wisp}
\end{equation}
under the condition $\|z-z^*b\|^2\leq \gamma n$. We also have
$$\frac{1}{n-1}|\bar{z}_j^*\bar{a}\bar{b}(z_j-z_j^*b)|\leq\frac{1}{n-1}|z_j-z_j^*b|\leq \frac{\sqrt{\gamma n}}{n-1}.$$
Therefore, as long as $|F_j|\vee|G_j|\vee|H_j|\leq\rho$, we have
$$\re{(\wt{z}_j\bar{z}_j^*\bar{a}\bar{b})}\geq 1-3(\sqrt{\gamma}+\rho).$$
By (\ref{eq:ratio-im-re}), we obtain the bound
\begin{eqnarray*}
|\wh{z}_j-z_j^*ab|^2 &\leq& \left|\frac{\im{(\wt{z}_j\bar{z}_j^*\bar{a}\bar{b})}}{\re{(\wt{z}_j\bar{z}_j^*\bar{a}\bar{b})}}\right|^2\mathbb{I}\left\{|F_j|\vee|G_j|\vee|H_j|\leq\rho\right\} \\
&& + 4\mathbb{I}\left\{|F_j|\vee|G_j|\vee|H_j|>\rho\right\} \\
&\leq& \frac{\left|\im\left(- \frac{1}{n-1}\bar{z}_j^*\bar{a}\bar{b}(z_j-z_j^*b)\right)+\im(F_j)+\im(G_j)+\im(H_j)\right|^2}{(1-3(\sqrt{\gamma}+\rho))^2} \\
&& + 4\mathbb{I}\{|F_j|>\rho\} + 4\mathbb{I}\{|G_j|>\rho\} + 4\mathbb{I}\{|H_j|>\rho\} \\
&\leq& \frac{1+\eta}{(1-3(\sqrt{\gamma}+\rho))^2}\left|\im(H_j)\right|^2 + \frac{3(1+\eta^{-1})}{(1-3(\sqrt{\gamma}+\rho))^2}|\im(F_j)|^2 \\
&& + \frac{3(1+\eta^{-1})}{(1-3(\sqrt{\gamma}+\rho))^2}\left|\im\left(- \frac{1}{n-1}\bar{z}_j^*\bar{a}\bar{b}(z_j-z_j^*b)\right)\right|^2 \\
&& + \frac{3(1+\eta^{-1})}{(1-3(\sqrt{\gamma}+\rho))^2}|\im(G_j)|^2 + 4\mathbb{I}\{|F_j|>\rho\} + 4\mathbb{I}\{|G_j|>\rho\} + 4\mathbb{I}\{|H_j|>\rho\},
\end{eqnarray*}
for some $\eta=o(1)$ which will be specified later. Here the last inequality is due to the fact that $(x_1 + x_2)^2 = x_1^2 + x_2^2 + 2(\eta^{1/2} x_1)(\eta^{-1/2}x_2) \leq (1+\eta)x_1^2 + (1+\eta^{-1})x_2^2$ for any $x_1,x_2\in\mathr$.

\paragraph{Step 3: Analysis of each error term.}

Next, we will analyze the error terms $F_j$, $G_j$ and $H_j$ separately. By triangle inequality, (\ref{eq:hollow1}) and (\ref{eq:hollow2}), we have
\begin{eqnarray}
\nonumber |F_j| &\leq& \frac{\left|\sum_{k\in[n]\backslash\{j\}}(A_{jk}-p)\bar{z}_k^*\bar{a}\bar{b}(z_k-{z}_k^*b)\right|}{\sum_{k\in[n]\backslash\{j\}}A_{jk}} \\
\nonumber && + \left|p\sum_{k\in[n]\backslash\{j\}}\bar{z}_k^*\bar{a}\bar{b}(z_k-{z}_k^*b)\right|\left|\frac{1}{\sum_{k\in[n]\backslash\{j\}}A_{jk}}-\frac{1}{(n-1)p}\right| \\
\nonumber &\leq& \frac{2}{np}\left|\sum_{k\in[n]\backslash\{j\}}(A_{jk}-p)\bar{z}_k^*(z_k-{z}_k^*b)\right| + p\sqrt{n}\|z-bz^*\|\frac{2\left|\sum_{k\in[n]\backslash\{j\}}(A_{jk}-p)\right|}{n^2p^2} \\
\nonumber &\leq& \frac{2}{np}\left|\sum_{k\in[n]\backslash\{j\}}(A_{jk}-p)\bar{z}_k^*(z_k-{z}_k^*b)\right| + C_1\frac{\sqrt{p\log n}}{np}\|z-bz^*\|.
\end{eqnarray}
Using (\ref{eq:hollow3}), we have
\begin{eqnarray*}
\sum_{j=1}^n|F_j|^2 &\leq& \frac{8}{n^2p^2}\sum_{j=1}^n\left|\sum_{k\in[n]\backslash\{j\}}(A_{jk}-p)\bar{z}_k^*(z_k-{z}_k^*b)\right|^2 + C_1^2\frac{\log n}{np}\|z-bz^*\|^2 \\
&\leq& \frac{8}{n^2p^2}\opnorm{A-\mathbb{E}A}^2\|z-bz^*\|^2 + C_1^2\frac{\log n}{np}\|z-bz^*\|^2 \\
&\leq& C_2\frac{\log n}{np}\ell(z,z^*).
\end{eqnarray*}
The above bound also implies
$$\sum_{j=1}^n\mathbb{I}\{|F_j|>\rho\}\leq \rho^{-2}\sum_{j=1}^n|F_j|^2\leq \frac{C_2}{\rho^2}\frac{\log n}{np}\ell(z,z^*).$$
Similarly, we can also bound the error terms that depend on $G_j$. By (\ref{eq:hollow1}) and (\ref{eq:hollow4}), we have
\begin{eqnarray*}
\sum_{j=1}^n|G_j|^2 &\leq& \frac{2\sigma^2}{n^2p^2}\sum_{j=1}^n\left|\sum_{k\in[n]\backslash\{j\}}A_{jk}W_{jk}(z_k-z_k^*b)\right|^2 \\
&\leq& \frac{2\sigma^2}{n^2p^2}\opnorm{A\circ W}^2\|z-bz^*\|^2 \\
&\leq& C_3\frac{\sigma^2}{np}\ell(z,z^*),
\end{eqnarray*}
and thus
$$\sum_{j=1}^n\mathbb{I}\{|G_j|>\rho\}\leq \rho^{-2}\sum_{j=1}^n|G_j|^2\leq \frac{C_3}{\rho^2}\frac{\sigma^2}{np}\ell(z,z^*).$$
For the contribution of $H_j$, we use (\ref{eq:hollow1}) and (\ref{eq:hollow5}), and have
\begin{eqnarray*}
\sum_{j=1}^n\mathbb{I}\{|H_j|>\rho\} &\leq& \sum_{j=1}^n\mathbb{I}\left\{\frac{2\sigma}{np}\left|\sum_{k\in[n]\backslash\{j\}}A_{jk}W_{jk}z_k^*\right|>\rho\right\} \\
&\leq& \frac{4\sigma^2}{\rho^2p} \exp\left(-\frac{1}{16}\sqrt{\frac{\rho^2 np}{\sigma^2}}\right).
\end{eqnarray*}
Next, we study the main error term $|\im(H_j)|^2$. By (\ref{eq:hollow1}), we have
\begin{align*}
\sum_{j=1}^n|\im(H_j)|^2 \leq& \left(1+C_4\sqrt{\frac{\log n}{np}}\right)^2\frac{\sigma^2}{n^2p^2}\sum_{j=1}^n\left|\sum_{k\in[n]\backslash\{j\}}A_{jk}\im(W_{jk}z_k^*\bar{z}_j^*\bar{a})\right|^2 \\
 \leq& (1+\eta)\left(1+C_4\sqrt{\frac{\log n}{np}}\right)^2\frac{\sigma^2}{n^2p^2}\sum_{j=1}^n\left|\sum_{k\in[n]\backslash\{j\}}A_{jk}\im(W_{jk}z_k^*\bar{z}_j^*)\right|^2 \\
& + (1+\eta^{-1})\left(1+C_4\sqrt{\frac{\log n}{np}}\right)^2\frac{\sigma^2}{n^2p^2}\sum_{j=1}^n\left|\sum_{k\in[n]\backslash\{j\}}A_{jk}\re(W_{jk}z_k^*\bar{z}_j^*)\right|^2|\im(\bar{a})|^2.
\end{align*}
By (\ref{eq:ori-wisp}), we have
$$
|\im(\bar{a})| = \frac{|\im(a_0)|}{|a_0|}\leq\frac{\frac{1}{n-1}\|z^*\|\|z-bz^*\|}{1-2\sqrt{\gamma}}\leq 4\sqrt{\gamma}.
$$
Together with (\ref{eq:hollow6}) and (\ref{eq:hollow7}), we have
$$\sum_{j=1}^n|\im(H_j)|^2 \leq \left(1+C_5\left(\eta+\eta^{-1}\gamma+\sqrt{\frac{\log n}{np}}\right)\right)\frac{\sigma^2}{2p}.$$
The last error term we need to analyze is $\left|\im\left(- \frac{1}{n-1}\bar{z}_j^*\bar{a}\bar{b}(z_j-z_j^*b)\right)\right|^2$. It can simply be bounded by $\frac{1}{(n-1)^2}|z_j-z_j^*b|^2$.

\paragraph{Step 4: Combining the bounds.} Plugging all the individual error bounds obtained in Step 3 into the error decomposition in Step 2, we obtain
\begin{eqnarray*}
\ell(\wh{z},z^*) &\leq& \sum_{j=1}^n|\wh{z}_j-z_j^*ab|^2 \\
&\leq& \left(1+C_6\left(\rho+\sqrt{\gamma}+\eta+\eta^{-1}\gamma+\sqrt{\frac{\log n}{np}}\right)\right)\frac{\sigma^2}{2p} \\
&& + \frac{16\sigma^2}{\rho^2 p}\exp\left(-\frac{1}{16}\sqrt{\frac{\rho^2np}{\sigma^2}}\right)  + C_6\left(\eta^{-1}+\rho^{-2}\right)\frac{\log n+\sigma^2}{np}\ell(z,z^*).
\end{eqnarray*}
We set
$$\eta=\sqrt{\gamma+\frac{\log n+\sigma^2}{np}}\quad\text{and}\quad \rho^2=\sqrt{32}\sqrt{\frac{\log n+\sigma^2}{np}}.$$
Then, since $\frac{\rho^2np}{\sigma^2}\rightarrow\infty$, we have
$$\frac{16\sigma^2}{\rho^2 p}\exp\left(-\frac{1}{16}\sqrt{\frac{\rho^2np}{\sigma^2}}\right)\leq \frac{\sigma^2}{\rho^2 p}\left(\frac{\sigma^2}{\rho^2 np}\right)^2\leq\frac{\sigma^2}{p}\sqrt{\frac{\sigma^2}{np}}.$$
Therefore, we have
$$\ell(\wh{z},z^*)\leq \left(1+C_7\left(\gamma^2+\frac{\log n+\sigma^2}{np}\right)^{1/4}\right)\frac{\sigma^2}{2p} + C_7\sqrt{\frac{\log n+\sigma^2}{np}}\ell(z,z^*).$$
Since the above inequality is derived from the condition (\ref{eq:hollow1}), (\ref{eq:hollow2}), (\ref{eq:hollow3}), (\ref{eq:hollow4}), (\ref{eq:hollow5}), (\ref{eq:hollow6}), (\ref{eq:hollow7}) and $\ell(z,z^*)\leq \gamma n$. It holds uniformly over all $z\in\mathbb{C}_1^n$ such that $\ell(z,z^*)\leq \gamma n$ with probability at least $1-(2n)^{-1}-\exp\left(-\left(\frac{np}{\sigma^2}\right)^{1/4}\right)$. The proof is complete.

\subsection{Proof of Lemma \ref{lem:lower-van-trees}}\label{sec:pf-low}

We use the version of van Trees' inequality presented as Equation (11) of the paper \cite{gill1995applications}. We have
\begin{eqnarray}
\nonumber && \inf_{\wh{T},\wh{S}}\int \int f(a)f(c)\mathbb{E}_{(a,c)}\left[(\wh{T}-T(a,c))^2+(\wh{S}-S(a,c))^2\right] \diff a \diff c \\
\label{eq:van-tree} &\geq& \int\int \Tr(J(a,c))f(a)f(c) \diff a \diff c - I(f),
\end{eqnarray}
where $I(f)$ is the information of the prior which will be elaborated later.
The matrix $J(a,c)$ is given by $J(a,c)=AB^{-1}A$, where
$$A=\frac{\partial(T,S)}{\partial(a,c)}=\begin{pmatrix}
c - \sqrt{1-c^2}\frac{a}{\sqrt{1-a^2}} & a-\sqrt{1-a^2}\frac{c}{\sqrt{1-c^2}} \\
-\frac{ac}{\sqrt{1-a^2}} - \sqrt{1-c^2} & \sqrt{1-a^2} + \frac{ac}{\sqrt{1-c^2}}
\end{pmatrix},$$
and $B$ is the Fisher information matrix. Due to the independence between $A_{12}Y_{12}$ and $\cbr{A_{1j}Y_{1j}}_{j\geq 3}\cup \cbr{A_{2j}Y_{2j}}_{j\geq 3}$, we have $B=B_1+B_2$
 with
$$B_1=\frac{2p}{\sigma^2}\begin{pmatrix}
\left|\frac{\partial T}{\partial a}\right|^2 + \left|\frac{\partial S}{\partial a}\right|^2 & \frac{\partial T}{\partial a}\frac{\partial T}{\partial c}  \\
\frac{\partial T}{\partial a}\frac{\partial T}{\partial c} & \left|\frac{\partial T}{\partial c}\right|^2 + \left|\frac{\partial S}{\partial c}\right|^2
\end{pmatrix},$$
and
$$B_2=\frac{2(n-2)p}{\sigma^2}\begin{pmatrix}
\frac{1}{1-a^2} & 0 \\
0 & \frac{1}{1-c^2}
\end{pmatrix},$$
In the following we show how to derive $\sbr{B_2}_{11}$. Let $(\xi,\zeta)^{\T}$ be a bivariate normal random vector distributed according to (\ref{eq:cond2}). Thus, the part of the likelihood function that involves $a$ is proportional to $l(\xi,\zeta,\cbr{A_{1j}}_{j\geq 3};a)$ which is defined as
\begin{align*}
l(\xi,\zeta,\cbr{A_{1j}}_{j\geq 3};a) =  \ebr{- \frac{\br{\xi - a\sum_{j\geq 3}A_{1j}}^2 + \br{\zeta -\sqrt{1-a^2}\sum_{j\geq 3}A_{1j} }^2}{\sigma^2 \sum_{j\geq 3}A_{1j}} }.
\end{align*}
Then
\begin{align*}
\sbr{B_2}_{11} = -\E \frac{\partial^2 \log  l(\xi,\zeta,\cbr{A_{1j}}_{j\geq 3};a)}{\partial a^2} = \E \frac{2\zeta}{\sigma^2 (1-a^2)^\frac{3}{2}} = \frac{2(n-2)p}{(1-a^2)\sigma^2}.
\end{align*}
The rest of the entries of $B_1$ and $B_2$ can be calculated analogously and are omitted here. 

We then have
\begin{eqnarray*}
\Tr(J(a,c)) &\geq& \Tr(AB_2^{-1}A^{\T})-|\Tr(A((B_1+B_2)^{-1}-B_2^{-1})A^{\T})| \\
&\geq& \fnorm{B_2^{-1/2}A^{\T}}^2\left(1-\fnorm{B_2^{1/2}(B_1+B_2)^{-1}B_2^{1/2}-I_2}\right).
\end{eqnarray*}
By direct calculation, $\fnorm{B_2^{-1/2}A^{\T}}^2=\frac{2\sigma^2}{(n-2)p}$.
Therefore,
\begin{eqnarray*}
&& \int\int\Tr(J(a,c))\pi(a)\pi(c) \diff a \diff c \\
&\geq& \frac{\sigma^2}{p(n-2)}\left(1-\int\int \fnorm{B_2^{1/2}(B_1+B_2)^{-1}B_2^{1/2}-I_2}f(a)f(c) \diff a \diff c\right).
\end{eqnarray*}
Now we set $f$ to be the density function of a distribution supported on the interval $[0.4,0.6]$. We require $f$ to be bounded by a constant, smooth and vanishes and the endpoints of the interval. We also require both $f'/f$ and $f'$ to be bounded by a constant on the interval $[0.4,0.6]$. This makes $f$ satisfy the regularity condition required by \cite{gill1995applications} so that van Trees' inequality (\ref{eq:van-tree}) holds. The existence of such a nice density function $f$ is guaranteed by a construction using mollifier function. By the choice of $f$, we know that both $a$ and $c$ are bounded away from $0$ and $1$. As a result, we have $\opnorm{B_2}\lesssim \frac{np}{\sigma^2}$, $\opnorm{B_2^{-1}}\lesssim \frac{\sigma^2}{np}$ and $\opnorm{(B_1+B_2)^{-1}}\lesssim \frac{\sigma^2}{np}$. Thus,
\begin{eqnarray*}
&& \fnorm{B_2^{1/2}(B_1+B_2)^{-1}B_2^{1/2}-I_2} \\
&\leq& \opnorm{B_2}\fnorm{(B_1+B_2)^{-1}-B_2^{-1}} \\
&\leq& \opnorm{B_2}\opnorm{B_2^{-1}}\opnorm{(B_1+B_2)^{-1}}\fnorm{B_1} \\
&\lesssim& n^{-1}.
\end{eqnarray*}
Hence, we have
$$\int\int\Tr(J(a,c))f(a)f(c) \diff a \diff c \geq \frac{\sigma^2}{(n-2)p}\left(1-C_1 n^{-1}\right).$$
Finally, we need to provide an upper bound for $I(f)$. The definition of $I(f)$ is given by
$$I(f)=\int_{[0.4,0.6]^2}\frac{1}{f(\theta_1)f(\theta_2)}\sum_{j,k,l\in\{1,2\}}\left(\frac{\partial}{\partial\theta_k}K_{jk}(\theta)f(\theta_1)f(\theta_2)\right)\left(\frac{\partial}{\partial\theta_l}K_{jl}(\theta)f(\theta_1)f(\theta_2)\right)\diff\theta,$$
where $\theta=(\theta_1,\theta_2)=(a,c)$ and $K(\theta)=AB^{-1}$, which is a $2\times 2$ matrix depending on the value of $\theta=(a,c)$. By the regularity conditions satisfied by the choice of $f$, we have
$$I(f)\leq C_2\max_{\theta\in[0.4,0.6]^2} \max_{j,k\in\{1,2\}}\left|\frac{\partial}{\partial\theta_k}K_{jk}(\theta)\right|^2.$$
The above bound can be explicitly calculated via the definitions of $A$, $B_1$ and $B_2$, but we omit the tedious details. Intuitively, the contribution of $B_1$ is negligible compared with that of $B_2$, and the contribution of $B_2$ is of order $np/\sigma^2$. Thus, $I(f)\leq C_3\left(\frac{\sigma^2}{np}\right)^2$. This leads to the lower bound
$$\inf_{\wh{T}}\int \int \pi(z_j)\pi(z_k)\mathbb{E}_z|\wh{T}-z_j\bar{z}_k|^2\diff z_j\diff z_k\geq \frac{\sigma^2}{np}\left(1-C\left(\frac{1}{n}+\frac{\sigma^2}{np}\right)\right).$$
The desired result is obtained by plugging the above lower bound to (\ref{eq:lamp}). This completes the proof.

\subsection{Proofs of Lemma \ref{lem:spec-ini} and Lemma \ref{lem:MLE-ini}}\label{sec:pf-ini}

The two lemmas can be proved via the same argument, and thus we present the proofs together. For any $z\in\mathbb{C}^n$ such that $\|z\|^2=n$, we have
$$\fnorm{p^{-1}A\circ Y -zz^{\H}}^2=\fnorm{p^{-1}A\circ Y}^2+n^2-2p^{-1}z^{\H}(A\circ Y)z.$$
Therefore, $\min_{\|z\|^2=n}\fnorm{p^{-1}A\circ Y -zz^{\H}}^2$ is equivalent to $\max_{\|z\|^2=n}z^{\H}(A\circ Y)z$. We can thus write $z^{(0)}$ as $z_j^{(0)}=\frac{\wh{z}_j}{|\wh{z}_j|}$, where
\begin{equation}
\wh{z}=\argmin_{\|z\|^2=n}\fnorm{p^{-1}A\circ Y -zz^{\H}}^2. \label{eq:OLS-spec}
\end{equation}
On the other hand, for any $z\in\mathbb{C}_1^n$,
\begin{eqnarray*}
&& \sum_{1\leq j< k\leq n}A_{jk}|Y_{jk}-z_j\bar{z}_k|^2 \\
&=& \sum_{1\leq j<k\leq n}A_{jk}(|Y_{jk}|^2+1) - \sum_{1\leq j<k\leq n}A_{jk}Y_{jk}(z_j\bar{z}_k+\bar{z}_jz_k) \\
&=& \sum_{1\leq j<k\leq n}A_{jk}(|Y_{jk}|^2+1) - z^{\H}(A\circ Y)z \\
&=& \sum_{1\leq j<k\leq n}A_{jk}(|Y_{jk}|^2+1)  - \frac{p}{2}\left(\fnorm{p^{-1}A\circ Y}^2+n^2-\fnorm{p^{-1}A\circ Y -zz^{\H}}^2\right).
\end{eqnarray*}
Therefore, $\min_{z\in\mathbb{C}_1^n}\sum_{1\leq j< k\leq n}A_{jk}|Y_{jk}-z_j\bar{z}_k|^2$ is equivalent to $\min_{z\in\mathbb{C}_1^n}\fnorm{p^{-1}A\circ Y -zz^{\H}}^2$, and the MLE can be equivalently written as
\begin{equation}
\wh{z}=\argmin_{z\in\mathbb{C}_1^n}\fnorm{p^{-1}A\circ Y -zz^{\H}}^2. \label{eq:OLS-MLE}
\end{equation}
Now we analyze (\ref{eq:OLS-spec}) and (\ref{eq:OLS-MLE}) simultaneously.  For $\wh{z}$ that is either (\ref{eq:OLS-spec}) or (\ref{eq:OLS-MLE}), we have
$$\fnorm{p^{-1}A\circ Y -\wh{z}\wh{z}^{\H}}^2 \leq \fnorm{p^{-1}A\circ Y -z^*z^{*\H}}^2.$$
Rearranging the above inequality, we have
$$\fnorm{\wh{z}\wh{z}^{\H}-z^*z^{*\H}}^2\leq 2\left|\Tr\left((\wh{z}\wh{z}^{\H}-z^*z^{*\H})(p^{-1}A\circ Y -z^*z^{*\H})\right)\right|,$$
which implies
\begin{equation}
\fnorm{\wh{z}\wh{z}^{\H}-z^*z^{*\H}}\leq 2\left|\Tr\left(\left(\frac{\wh{z}\wh{z}^{\H}-z^*z^{*\H}}{\fnorm{\wh{z}\wh{z}^{\H}-z^*z^{*\H}}}\right)(p^{-1}A\circ Y -z^*z^{*\H})\right)\right|. \label{eq:airpods}
\end{equation}
Since $\frac{\wh{z}\wh{z}^{\H}-z^*z^{*\H}}{\fnorm{\wh{z}\wh{z}^{\H}-z^*z^{*\H}}}$ is a Hermitian matrix of rank at most two, it has the spectral decomposition
$$\frac{\wh{z}\wh{z}^{\H}-z^*z^{*\H}}{\fnorm{\wh{z}\wh{z}^{\H}-z^*z^{*\H}}}=\lambda_1 uu^{\H} + \lambda_2 vv^{\H},$$
where $u,v$ are complex unit vectors orthogonal to each other and $\lambda_1,\lambda_2$ are real and satisfy $\lambda_1^2+\lambda_2^2=1$. Then, we bound the right hand side of (\ref{eq:airpods}) by
\begin{eqnarray*}
&& 2|\lambda_1|\left|u^{\H}(p^{-1}A\circ Y -z^*z^{*\H})u\right| + 2|\lambda_2|\left|v^{\H}(p^{-1}A\circ Y -z^*z^{*\H})v\right| \\
&\leq& 2(|\lambda_1|+|\lambda_2|)\opnorm{p^{-1}A\circ Y -z^*z^{*\H}} \\
&\leq& 2\sqrt{2}\opnorm{p^{-1}A\circ Y -z^*z^{*\H}}.
\end{eqnarray*}
Therefore,
\begin{eqnarray*}
\fnorm{\wh{z}\wh{z}^{\H}-z^*z^{*\H}} &\leq& 4\sqrt{2}\opnorm{p^{-1}A\circ Y -z^*z^{*\H}} \\
&\leq& \frac{1}{p}\opnorm{(A-\mathbb{E}A)\circ z^*z^{*\H}} + \frac{\sigma}{p}\opnorm{A\circ W}.
\end{eqnarray*}
By Lemma \ref{lem:ER-graph},
\begin{eqnarray*}
\opnorm{(A-\mathbb{E}A)\circ z^*z^{*\H}} &=& \sup_{\|u\|=1}\left|\sum_{1\leq j\neq k \leq n}(A_{jk}-p)z_j^*\bar{z}^*_ku_j\bar{u}_k\right| \\
&\leq& \opnorm{A-\mathbb{E}A} \\
&\leq& C_1\sqrt{np},
\end{eqnarray*}
with probability at least $1-n^{-10}$.
By Lemma \ref{lem:bandeira}, $\opnorm{A\circ W}\leq C_2\sqrt{np}$ with probability at least $1-n^{-10}$. Thus,
$$
\fnorm{\wh{z}\wh{z}^{\H}-z^*z^{*\H}}^2 \leq C_3\frac{n(\sigma^2+1)}{p},
$$
with probability at least $1-2n^{-10}$. For the MLE $\wh{z}$, it satisfies $\wh{z}\in\mathbb{C}_1^n$, and thus we can use Lemma \ref{lem:loss-equiv} and obtain the bound
$$\ell(\wh{z},z^*)\leq \frac{1}{n}\fnorm{\wh{z}\wh{z}^{\H}-z^*z^{*\H}}^2 \leq C_3\frac{\sigma^2+1}{p},$$
with probability at least $1-2n^{-10}$. For $z^{(0)}$, its definition implies that
$$z_j^{(0)}\bar{z}_k^{(0)}=\frac{\wh{z}_j\bar{\wh{z}}_k}{|\wh{z}_j\bar{\wh{z}}_k|},$$
and therefore, we have
$$|z_j^{(0)}\bar{z}_k^{(0)}-z_j^*\bar{z}_k^*|\leq 2|\wh{z}_j\bar{\wh{z}}_k-z_j^*\bar{z}_k^*|,$$
by Lemma \ref{lem:trivial}.
Use this inequality, and we have
$$\fnorm{z^{(0)}z^{(0)\H}-z^*z^{*\H}} \leq 2\fnorm{\wh{z}\wh{z}^{\H}-z^*z^{*\H}},$$
and thus $\fnorm{z^{(0)}z^{(0)\H}-z^*z^{*\H}}^2\leq C_4\frac{n(\sigma^2+1)}{p}$ with probability at least $1-2n^{-10}$. By Lemma \ref{lem:loss-equiv}, we have
$$\ell(z^{(0)},z^*)\leq \frac{1}{n}\fnorm{z^{(0)}z^{(0)\H}-z^*z^{*\H}}^2 \leq C_4\frac{\sigma^2+1}{p},$$
with probability at least $1-2n^{-10}$. The proof is complete.

\bibliographystyle{dcu}
\bibliography{reference}

\end{document}